\documentclass[11pt]{article}
\usepackage{amsmath,amsfonts,amssymb,amsthm}

\font\Bbb=msbm10
\def\R{\hbox{\Bbb R}}

\newcount\diagramnumber

\def\address#1
{
	\expandafter\def\expandafter\@aabuffer\expandafter
	{\@aabuffer{\affiliationfont{#1}}\relax\par
	\vspace*{13pt}}
}

\newtheorem{theorem}{Theorem}[section]
\newtheorem{lemma}[theorem]{Lemma}

\newtheorem{remark}[theorem]{Remark}
\newtheorem{example}[theorem]{Example}

\def\prbox
{
	\par
	\vskip-\lastskip\vskip-\baselineskip\hbox to \hsize{\hfill\fboxsep0pt\fbox{\phantom{\vrule width5pt height5pt depth0pt}}}
}

\usepackage{epsfig}

\begin{document}


\title{Doodles and commutator identities
}

\author{Andrew Bartholomew \\ 
School of Mathematical Sciences, University of Sussex\\
Falmer, Brighton, BN1 9RH, England\\
e-mail address: andrewb@layer8.co.uk \\ 
\\
Roger Fenn \\ 
School of Mathematical Sciences, University of Sussex\\
Falmer, Brighton, BN1 9RH, England\\
e-mail address: roger.fenn@gmail.com \\ 
\\ 
Naoko Kamada \\ 
Graduate School of Natural Sciences, Nagoya City University\\
Nagoya, Aichi 467-8501, Japan\\
e-mail address: kamada@nsc.nagoya-cu.ac.jp \\ 
\\ 
Seiichi Kamada \\ 
Department of Mathematics, Osaka University\\
Toyonaka, Osaka 560-0043, Japan\\
e-mail address: kamada@math.sci.osaka-cu.ac.jp
}

\maketitle

\begin{abstract}
A doodle is a collection of immersed circles without triple intersections in the $2$-sphere. 
It was shown by the second author and P.~Tayler that doodles induce commutator identities (identities amongst commutators) in a free group.  
In this paper we observe this idea more closely by concentrating on doodles with proper noose systems and elementary commutator identities.  In particular we show that there is a bijection between cobordism classes of colored doodles and weak equivalence classes of elementary commutator identities.  
\end{abstract}



\vspace{0.3cm}
\begin{center}
{\LARGE{ Dedicated to the memory of Patrick~Dehornoy}}
 \end{center}
 
\section{Introduction}

Doodles were first introduced by the second author and P.~Taylor in \cite{FT}. The original definition of a doodle was a collection of embedded circles in the $2$-sphere $S^2$ whose multiple points are transverse double points. 
M.~Khovanov, \cite{MK}, extended the idea to allow each component to be an immersed circle in $S^2$ or a closed oriented surface. In \cite{BFKKdoodles} the authors introduced virtual doodles as an analogy of virtual knots \cite{Kauf} and  
developed  stably equivalence classes of doodles on closed oriented surfaces with a relationship with virtual doodles, which is analogue to a relationship between virtual knots and stably equivalence classes of knot diagrams on closed oriented surfaces given in \cite{CKS, KK}. 

In this paper, by a {\em doodle diagram} or a {\em diagram} we mean a collection of immersed circles in $S^2$ whose multiple points are transverse double points.  
(Double points of a diagram are also referred to as {\em crossings}.) 
   Two diagrams are said to be {\em equivalent} if they are equivalent under the equivalence relation generated by ambient isotopies in $S^2$ and local moves depicted in Figure~\ref{dfig1and2}, where 
$H_1^+$ generates a monogon, $H_1^-$ deletes it, $H_2^+$ generates a bigon and $H_2^-$ deletes it. 
   The equivalence class is called a {\em doodle}.  As is the usual custom, we will often not distinguish between a doodle and its diagram.
 A doodle or a doodle diagram is {\em oriented} if each component is oriented. Throughout this paper we assume that doodles and doodle diagrams are oriented.

    \begin{figure}[h]
    \centerline{\epsfig{file=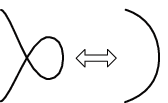, height=1.5cm} 
    \qquad \qquad 
    \epsfig{file=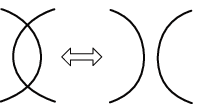, height=1.5cm} }
    \caption{$H_1^{\pm 1}$ and $H_2^{\pm 1}$}\label{dfig1and2}
    \end{figure}

Figure~\ref{dfigPoppydfigBorromean} shows diagrams of two doodles (without orientations). The first, called the {\em poppy}, has one component and the other with 3 components is called the {\em Borromean doodle}.

    \begin{figure}[h]
    \centerline{\epsfig{file=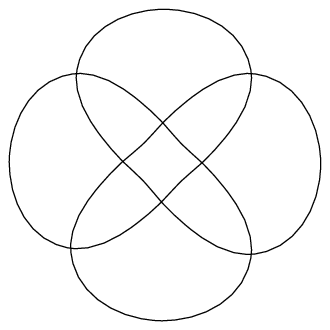, height=2.5cm} 
    \qquad \qquad 
    \epsfig{file=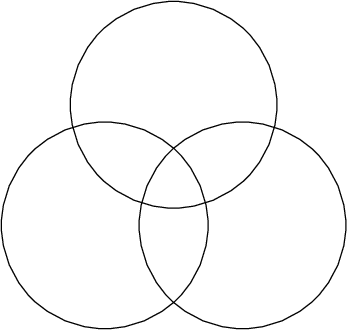, height=2.5cm}}
    \caption{The poppy and the Borromean doodle}\label{dfigPoppydfigBorromean}
    \end{figure}

A diagram is called {\em minimal} if there are no monogons and no bigons, or equivalently if neither  
$H_1^-$ nor $H_2^-$ can be applied to the diagram.  Any doodle has a unique minimal diagram modulo trivial components in the sense of Corollary~2.8.9 of \cite{Fenn83} and Theorem~2.2 of \cite{MK}.  (Refer to 
\cite{BFKKdoodles} for minimal diagrams of doodles on closed oriented surfaces and their stably equivalence classes.)  The diagrams in Figure~\ref{dfigPoppydfigBorromean} are minimal diagrams. 

A doodle or a doodle diagram is called {\em colored} or {\em $S$-colored} when each component is labeled by an element of a fixed non-empty set $S$.  

The second author and Taylor showed in \cite{FT} that doodles induce commutator identities (identities amongst commutators) in the free group. 
Precisely speaking, for a non-empty set $S$, an $S$-colored doodle diagram with a noose system yields a commutator identity in the free group on $S$.  We will recall this in the next section. 
For example consider the following examples from the Borromean doodle (Example~\ref{example:A}).

The left hand side of Figure~\ref{dfigBorromean1dfigBorromean2} yields the identity
$$(a,c)^b (b,c) (b,a)^c (c,a) (c,b)^a (a,b)\equiv1$$
where for example, $(a,c)^b=b^{-1}(a^{-1}c^{-1}ac)b$, and the right hand side yields
$$(bc,a) (ca,b) (ab,c) \equiv1.$$

    \begin{figure}[h]
    \centerline{\epsfig{file=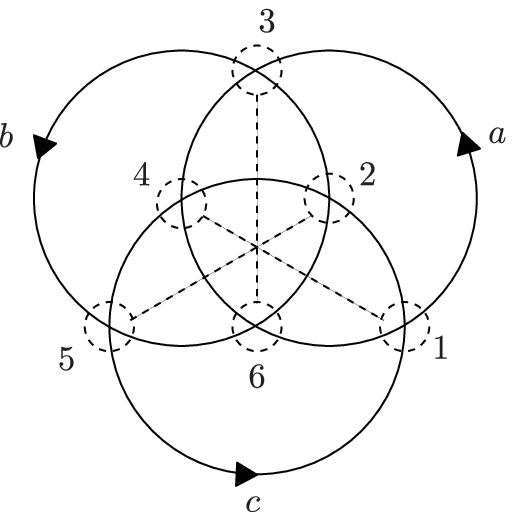, height=5cm} 
    \qquad \qquad 
    \epsfig{file=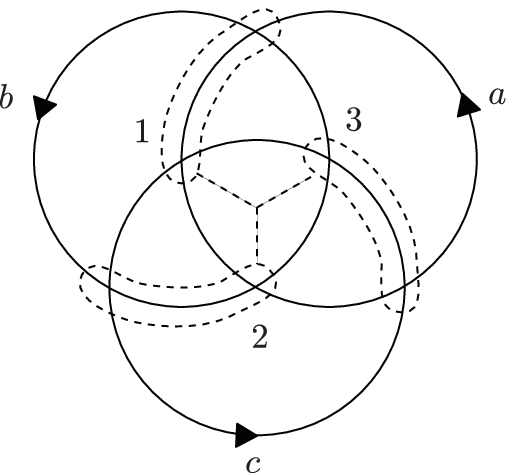, height=4.8cm} }
    \caption{Doodles with noose systems}\label{dfigBorromean1dfigBorromean2}
    \end{figure}

Another example which can be extracted from the Borromean doodle is the Hall-Witt identity (cf. \cite{H}): 
$$((a,b),c^a)  ((c,a),b^c) ((b,c),a^b) \equiv 1.$$
This is a group-theoretic analogue of the Jacobi identity for Lie algebras. 
These three examples from the Borromean doodles were introduced in \cite{FT}.  

In Sections~\ref{sect:CommutatorIdentities} and \ref{sect:Diagrams}, 
we first recall the idea in \cite{FT} to obtain a commutator identity from a colored doodle diagram using a noose system  (Theorem~\ref{thm:IdentityGeneral}), and conversely we show that for a commutator identity there exists a colored doodle diagram and a noose system which yield the given commutator identity (Theorem~\ref{thm:DiagramGeneral}).  
   In particular, when we use a proper noose system, we obtain an elementary commutator identity (Theorem~\ref{thm:IdentityProper}). Conversely, for an elementary commutator identity there exists a colored doodle diagram and a proper noose system which yield the given elementary commutator identity (Theorem~\ref{thm:DiagramProper}). The definitions of a (proper) noose system and an (elementary) commutator identity are given in Section~\ref{sect:CommutatorIdentities}.  

After Section~\ref{sect:EquivalnceIdentities} we concentrate to colored 
 doodles with proper noose systems and elementary commutator identities. 
Then the relationship between colored doodles and commutator identities established in \cite{FT} and in Sections~\ref{sect:CommutatorIdentities} and \ref{sect:Diagrams} 
becomes clear to understand with 
the action of the braid group and other fundamental transformations.  
We also discuss cobordisms of colored doodles.  

In Section~\ref{sect:EquivalnceIdentities} we introduce three kinds of equivalence relations on elementary commutator identities: strict equivalence $\cong$, equivalence $\simeq$ and weak equivalence $\sim$.  
We show that a colored doodle diagram induces a unique elementary commutator identity up to equivalence $\simeq$
 (Theorem~\ref{thm:41}) 
 and that a colored doodle induces a unique elementary commutator identity up to equivalence $\sim$  (Theorem~\ref{thm:42}) . 
 
In Section~\ref{sect:Cobordism} we introduce the notion of cobordism for colored doodles and doodle diagrams. Then two colored doodle diagrams are cobordant if and only if they induce weakly equivalent elementary commutator identities (Theorem~\ref{thm:CobC}).   There is a bijection between cobordism classes of colored doodles and weak equivalence classes of elementary commutator identities (Theorem~\ref{thm:CobD}).

This work was supported by JSPS KAKENHI Grant Numbers JP19K03496 and JP19H01788.

\section{How to obtain a commutator identity}
\label{sect:CommutatorIdentities}

We recall the idea in \cite{FT} to obtain a commutator identity from a doodle diagram.  
Throughout this paper we identify $S^2$ with $\R^2 \cup \{\infty\}$.  
For a doodle diagram $D$, we denote by $\Sigma(D)$ the set of the crossings. 
When $D$ meets $\infty$, moving it slightly, we assume that $D$ is away from $\infty$  
 so that we can draw it in $\R^2= S^2 \setminus\{\infty\}$.

Let $X = D^2 \cup I'$, where $D^2 = 
\{ (x,y) \in \R^2 \mid x^2 + y^2 \leq 1\}$ and $I'=\{ (x,y) \in \R^2 \mid x \in [1,2], y=0 \}$.  

A {\em noose} is the image $N$ of an embedding of $X$ in $S^2$. 
The {\em loop} (or the {\em rope}) of a noose is the image of $S^1 = \partial D^2$ (or the image of $I'$), and the   
{\em head} is the image of $D^2$.  
The {\em neck} is the image of $(1,0)$, where the head is joined to the rope, and   
the {\em root} is the image of $(2,0)$.  
Unless otherwise stated, we always assume that $N \subset \R^2= S^2 \setminus\{\infty\}$.

We usually present a noose by drawing its loop and rope in $\R^2$, where the head is understood to be the bounded region by the loop.  Moreover, we assume that the loop is oriented counterclockwise.  

By a {\em noose for a doodle diagram $D$} we mean a noose $N$ whose root misses $D$ such that 
the loop and the rope are transverse to the components of $D$ and 
disjoint from the crossings of $D$. 

A {\em noose system} is a sequence of nooses ${\cal N}= (N_1, \dots, N_m)$ such that (i) they are mutually disjoint except their common roots called the {\it base point}, and (ii) the ropes of them appear counterclockwise in this order around the base point.  

A {\em noose system for a doodle diagram $D$} means a noose system consisting of nooses for $D$ 
such that  every crossing of $D$ is in a head.  
For examples, see Figure~\ref{dfigBorromean1dfigBorromean2}, where the numbers $1, 2, \dots$ indicate the ordering of the nooses.

Let $S=\{ a, b, c, \dots\}$ be a non-empty set, which we use for labeling a doodle diagram. 
We denote by $S^{-1} = \{ a^{-1}, b^{-1}, c^{-1}, \dots \}$ the set of inverse letters of $S$, and by 
${\rm Word}(S \cup S^{-1})$ the monoid of words on $S \cup S^{-1}$.   For words 
$u, v \in {\rm Word}(S \cup S^{-1})$, $(u,v)$ means $u^{-1} v^{-1} u v $ and $u^v$ means $v^{-1} u v$.

A doodle or a doodle diagram is said to be  {\em colored} or {\em $S$-colored} if every component is labeled by an element of  a fixed non-empty set $S$.  

Let $\alpha$ be a path transverse to the components of an $S$-colored doodle diagram $D$.  For an intersection $x$ of $\alpha$ and $D$, the {\em intersection letter of $\alpha$  with $D$ at $x$} is an element $a^{\epsilon} \in S \cup S^{-1}$ 
such that $a \in S$ is the label of the component of $D$ where the intersection $x$ occurs and 
the exponent $\epsilon$ is $+1$ if the path $\alpha$ passes across the component of $D$ from the right hand side 
of the component with respect to its orientation 
to the left hand side, or $\epsilon$ is $-1$ otherwise.  
See Figure~\ref{dfigIWadfigIWb}.  
The {\em intersection word along $\alpha$ with $D$} is a word on $S \cup S^{-1}$ obtained by reading off the intersection letters at the intersection points along $\alpha$.

    \begin{figure}[h]
    \centerline{\epsfig{file=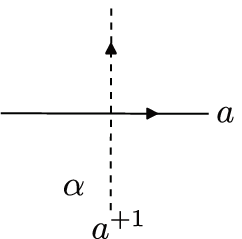, height=3cm} 
    \qquad \qquad 
    \epsfig{file=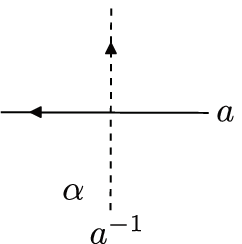, height=3cm} }
    \caption{Intersection letters}\label{dfigIWadfigIWb}
    \end{figure}

Let $D$ be a colored doodle diagram and $N$ a noose for $D$.  Consider a path $\ell_N : [0,1] \to S^2$ which starts from the root, goes along the rope, turns  along the loop of the noose counterclockwise, comes back to the root along the rope.  
The {\it intersection word of the noose $N$ with $D$} is the intersection word along $\ell_N$ with $D$.

When we say a {\em commutator identity} on $S$ or in the free group on $S$, it is an expression 
 $$w_1 \dots w_m \equiv 1$$ 
 such that each $w_i$ is a word on $S \cup S^{-1}$ representing an element of the commutator subgroup of the free group and the product $w_1 \dots w_m$ is a word representing the identity element of the free group.

\begin{theorem}[Fenn and Taylor \cite{FT}]\label{thm:IdentityGeneral} 
Let $D$ be a colored oriented doodle diagram.  Let ${\cal N}=(N_1, \dots, N_m)$ be a noose system for  $D$, and   
let $w_1, \dots, w_m$ be the intersection words of $N_1, \dots, N_m$ with $D$, respectively.  Then 
the following commutator identity holds in the free group on $S$: 
 $$w_1 \dots w_m \equiv 1.$$ 
\end{theorem}

We call the commutator identity in Theorem~\ref{thm:IdentityGeneral} the 
{\it commutator identity obtained from $D$ by using ${\cal N}$} and denote it by 
$$ I(D, {\cal N}). $$

A basic observation justifying Theorem 2.1 is the following.  The proof may be left to the reader.

\begin{lemma}\label{thm:BasicObservation}
Consider a 2-disk $B$ with boundary $\alpha$ and suppose there is a family of properly embedded and pairwise disjoint oriented arcs in $B$ labelled by elements of $S$.  Then, starting at any point in $\alpha$ and reading the word in $S \cup S^{-1}$
corresponding to the intersections of $\alpha$ with the endpoints of the arcs, according to Figure 4, as one travels around $\alpha$ in either direction back to the base point, the resulting word represents the identity element of the free group on $S$.  Conversely, given a finite set of points on $\alpha$ labelled by $S \cup S^{-1}$, if the resulting word obtained by traversing around $\alpha$ represents the identity on the free group on $S$, then there exists a finite set of oriented arcs properly and disjointly embedded in $B$, labelled by $S$, whose endpoints are the given points on $\alpha$, with the given labels.
\end{lemma}

\begin{example}\label{example:A}{\rm 
(1) 
Let $D$ and ${\cal N}$ be a colored doodle diagram and a noose system 
on the left hand side of Figure~\ref{dfigBorromean1dfigBorromean2}, where $S=\{a,b,c\}$.  
The intersection word of the first noose  is $b^{-1} a^{-1} c^{-1} a c b $, which is denoted by $b^{-1} (a,c) b$ or $(a,c)^b$ in our notation.   
The intersection words of the other nooses on the left hand side are 
$(b,c)$, $(b,a)^c$, $(c,a)$, $(c,b)^a$, and $(a,b)$, respectively.  Combining these words,  
we obtain the commutator identity $I(D, {\cal N})$ in the free group on $S$:  
$$(a,c)^b \, (b,c) \, (b,a)^c \, (c,a) \, (c,b)^a \, (a,b) \equiv 1.$$

(2) 
Let $D$ and ${\cal N}$ be a colored doodle diagram and a noose system 
on the right hand side of Figure~\ref{dfigBorromean1dfigBorromean2}.  
The intersection word of the first noose  is $c^{-1} b^{-1} a^{-1} b c a$, which is denoted by $(bc, a)$.  
The intersection words of the second noose and the third one are 
$(ca, b)$ and $(ab, c)$.  Then we have the commutator identity $I(D, {\cal N})$: 
$$(bc,a) \, (ca,b) \, (ab,c) \equiv1.$$ 
}\end{example}

\begin{remark}{\rm 
In Theorem~\ref{thm:IdentityGeneral}, 
when we change the coloring of $D$ by using a permutation $\sigma: S \to S$,  the commutator identity $I(D, {\cal N})$ is changed by replacing elements of $S$ by the corresponding elements under $\sigma$. 
For a subset $T$ of $S$, when we reverse the orientation of every component whose label is in $T$, the commutator  identity $I(D, {\cal N})$  is changed by switching  elements $x$ and $x^{-1}$ for every $x \in T$.  
}\end{remark}

A commutator identity in the free group on $S$ is called {\em elementary}, {\em elementary on $S$} or 
{\em $S$-elementary}  if it is written in a form such that 
$$ (a_1, b_1)^{u_1} \cdots (a_m, b_m)^{u_m} \equiv 1, $$ 
where  $a_i, b_i \in S$ and $u_i \in {\rm Word}(S \cup S^{-1})$ for each $i =1, \dots, m$.

A noose for a doodle diagram $D$ is {\it proper} if it satisfies the following two conditions (P1) and (P2): 
\begin{itemize}
\item[(P1)] 
The intersection of $D$ and the head is a union of two embedded arcs which intersect each other on a single point. 
\item[(P2)] 
The loop passes across a component of $D$ from left to right with respect to the orientation of the component, across a component of $D$ from left to right, then across the former component from   
right to left and across the latter component from right to left.   
(See Figure~\ref{dfigProper}.)
\end{itemize}

    \begin{figure}[h]
    \centerline{\epsfig{file=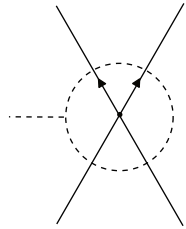, height=3cm} }
    \caption{The head of a proper noose for a doodle diagram}\label{dfigProper}
    \end{figure}

A noose system for a doodle diagram $D$ is {\it proper} if each noose is proper.  
The noose system depicted on the left hand side of Figure~\ref{dfigBorromean1dfigBorromean2} is proper.  
For any doodle diagram, there exists a proper noose system.

\begin{theorem}\label{thm:IdentityProper} 
In the situation of Theorem~\ref{thm:IdentityGeneral}, suppose that the noose system ${\cal N}$ is proper.  Then the commutator identity $I(D, {\cal N})$ obtained from $D$ by using ${\cal N}$ is elementary.  
\end{theorem}

\begin{proof}
Let ${\cal N}= (N_1, \dots, N_m)$ be a proper noose system for $D$. 
For each $i=1, \dots, m$, the intersection word along the loop of $N_i$ with $D$ is 
 $a_i^{-1} b_i^{-1} a_i b_i = (a_i, b_i)$ for some $a_i, b_i \in S$.  Let $u_i^{-1}$ be
  the intersection word along the rope of $N_i$ with $D$, where we regard the rope as a path from the base point to the neck.   Then the intersection word of $N_i$ is 
 $(a_i, b_i)^{u_i}$.  Thus the commutator identity is elementary.  
\end{proof}

\section{Constructing a doodle diagram from a commutator identity}
\label{sect:Diagrams}

In this section we show that every commutator identity is obtained from a doodle diagram, namely, 
for a given commutator identity, there exists a (non-unique) colored doodle diagram and a noose system which yield the commutator  identity. For an elementary commutator identity, we can take a proper noose system.

For a compact surface $M$ in $S^2$, a {\em doodle diagram in $M$} or {\em over $M$} means a collection of immersed circles  and properly immersed arcs in $M$ whose multiple points are transverse double points.  (When $M=S^2$, it is a doodle diagram in the usual sense.)

We first demonstrate, using an example, our method of constructing a doodle diagram and a noose system from a commutator identity.   
Let $S = \{a,b,c\}$ and consider the Hall-Witt identity, \cite{H}:  
$$((a,b),c^a) \, ((c,a),b^c)  \, ((b,c),a^b) \equiv 1. $$

The left hand side consists of three commutators, 
$((a,b),c^a)$, $((c,a),b^c)$, and $((b,c),a^b)$. 
Prepare a noose system consisting of three nooses, $N_1, N_2$ and $N_3$.  
For simplicity, we draw a figure such that each head is a rectangle as in the left hand side of Figure~\ref{dfigBCab}.  
For the first head, we draw horizontal and vertical parallel lines with labels in $S$ and with orientations such that the intersection word along the loop of $N_1$ is 
$$((a,b),c^a)  = (b^{-1} a^{-1} b a)(a^{-1} c^{-1} a) (a^{-1} b^{-1} a b)(a^{-1} c a)$$  
as in the figure.  
Similarly, we draw horizontal and vertical lines with labels and with orientations for the second and the third heads such that the 
intersection words along their loops are $((c,a),b^c)$ and $((b,c),a^b)$, respectively.  
Let $A$ be a regular neighborhood of the union of the three nooses, which is a $2$-disk in $S^2$.  We define  a doodle diagram over $A$, denoted by $D \cap A$, to be the union of the horizontal and vertical lines constructed for the heads.  The intersection word along the boundary of $A$ with $D$ is 
$((a,b),c^a) \,  ((c,a),b^c) \, ((b,c),a^b)$.  Since this word represents the identity element of the free group on $S$, 
by Lemma~\ref{thm:BasicObservation}, 
we can extend the diagram  $D \cap A$ to a colored doodle diagram $D$ in $S^2$ by adding some simple arcs  in the closure of $S^2 \setminus A$.  See the right hand side of Figure~\ref{dfigBCab}.  
Then we have a colored doodle diagram $D$ and a noose system ${\cal N}$ such that 
$I(D, {\cal N})$ is the given commutator identity.  

    \begin{figure}[h]
    \centerline{\epsfig{file=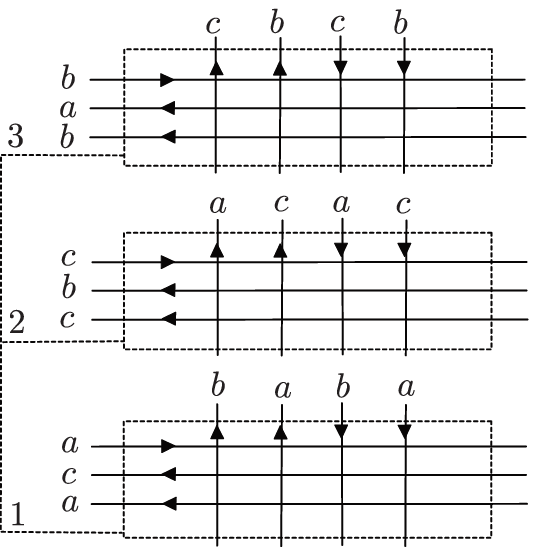, height=5.5cm} 
    \qquad \qquad 
    \epsfig{file=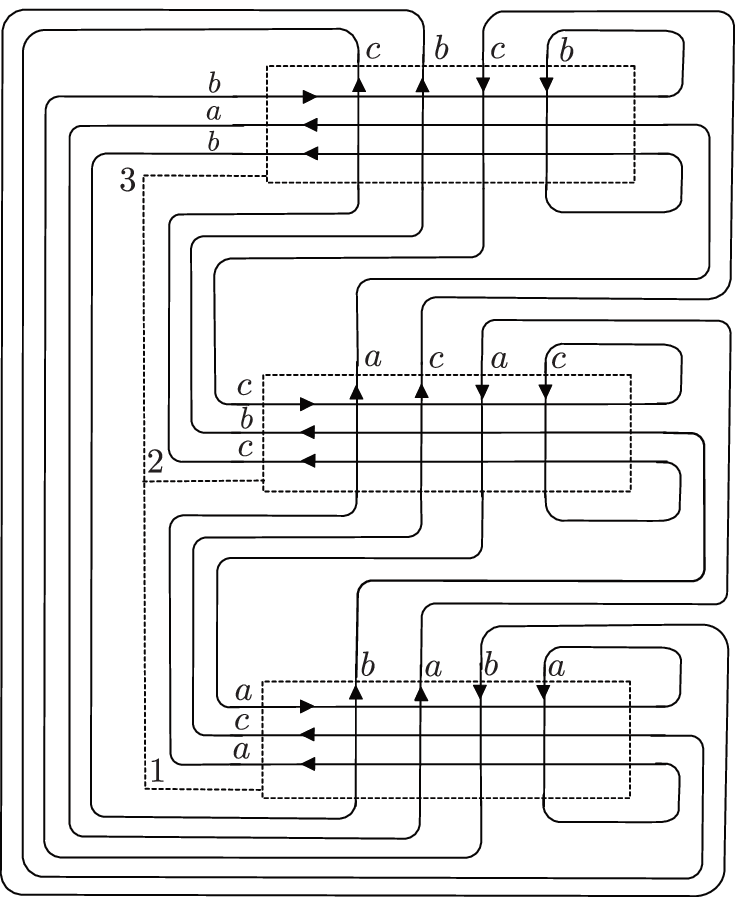, height=8cm} }
    \caption{Construction of a diagram}\label{dfigBCab}
    \end{figure}

The method explained above is always applicable to any commutator identity which is 
written in a form that 
$$ (s_1, t_1) \cdots (s_m, t_m) \equiv 1, $$ 
where  $s_i, t_i \in {\rm Word}(S \cup S^{-1})$ for each $i =1, \dots, m$. 

Modifying the method, we obtain the following. 

\begin{lemma}\label{thm:DiagramSpecial}
For a given commutator identity  in the free group on $S$ which is 
written in a form that 
$$ (s_1, t_1)^{u_1} \cdots (s_m, t_m)^{u_m} \equiv 1, $$ 
where  $s_i, t_i, u_i \in {\rm Word}(S \cup S^{-1})$ for each $i =1, \dots, m$, 
 there is a colored doodle diagram $D$ and a noose system ${\cal N}$ such that $I(D, {\cal N})$ is the given commutator identity.  
\end{lemma}

\begin{proof}
We modify the argument above. 
Let $w_1 \cdots w_m \equiv 1$ be the commutator identity with 
$w_i = (s_i, t_i)^{u_i}$ for some $s_i, t_i, u_i \in  {\rm Word}(S \cup S^{-1})$ for each $i=1, \dots, m$. 
   Prepare a noose system consisting of $m$ nooses, $N_1, \dots, N_m$.  
We draw a figure such that each head is a rectangle as before.  
For each noose $N_i$,  draw  horizontal parallel lines  
and vertical parallel lines with labels and orientations in a regular neighborhood of the head such that the intersection word along the loop of $N_i$ is $(s_i, t_i)$.   
And draw some small arcs intersecting the rope of $N_i$ transversely equipped with labels in $S$ and with orientations such that the intersection word along the rope from the base point to the neck is $u_i^{-1}$.

Let $A$ be a regular neighborhood of the union of the nooses, which is a $2$-disk in $S^2$.  We assume that the horizontal lines, vertical lines and the small arcs intersecting the rope of $N_i$ are properly embedded arcs in $A$.  
    We define a doodle diagram over $A$, denoted by $D \cap A$,  to be the union of these arcs with labels and orientations.  
For each $i$, the intersection word of $N_i$ is $w_i = (s_i, t_i)^{u_i}$.  
Since $w_1 \cdots w_m$ represents the identity in the free group, 
by Lemma~\ref{thm:BasicObservation}, 
we can extend the diagram $D \cap A$ in $A$ to a doodle diagram in $S^2$ by adding some simple arcs in the closure of $S^2 \setminus A$.   
Then we obtain a desired colored doodle diagram and a noose system.  
\end{proof} 

\begin{theorem}[cf. \cite{FT}]\label{thm:DiagramProper}
Every elementary commutator identity is obtained from a doodle diagram.  Namely, 
for any elementary commutator identity in the free group on $S$, 
$$ (a_1, b_1)^{u_1} \cdots (a_m, b_m)^{u_m} \equiv 1, $$ 
where $a_i, b_i \in S$ and $u_i \in {\rm Word}(S \cup S^{-1})$ for each $i =1, \dots, m$, 
 there exists a colored doodle diagram $D$ and a proper noose system ${\cal N}$ such that $I(D, {\cal N})$ is the given elementary commutator identity.  
\end{theorem}

\begin{proof}
Apply the same argument with the proof of Lemma~\ref{thm:DiagramSpecial}. Note that for each noose 
$N_i$, we draw a horizontal line with label $b_i$ 
and a vertical line with label $a_i$ with orientations in a regular neighborhood of the head such that the intersection word along the loop is $(a_i, b_i)$. Then the noose system is proper.   
\end{proof}

It is mentioned in \cite{FT} without details that a similar result to 
Theorem~\ref{thm:DiagramProper} is obtained by Rourke's {\em standard diagram} \cite{Rourke} for a certain $2$-dimensional C.W. complex. Our argument gives a method of construction of a doodle diagram and a noose system.

\begin{theorem}\label{thm:DiagramGeneral}
Every commutator identity is obtained from a doodle diagram. Namely, 
for any commutator identity in the free group on $S$,   
there exists a colored doodle diagram $D$ and a noose system ${\cal N}$ such that $I(D, {\cal N})$ is the given commutator identity. 
\end{theorem}

\begin{proof}
Let $w_1 \dots w_m \equiv 1$ be a commutator identity. 
Prepare a noose system consisting of $m$ nooses, $N_1, \dots, N_m$.  
Fix $i  \in \{1, \dots, m\}$ and let $w = w_i$ and $N= N_i$. Let $B$ denote the head of the noose $N$. 
Since $w$ represents an element of the commutator subgroup, there exists a word 
$W= (a_1, b_1)^{u_1} \cdots (a_k, b_k)^{u_k}$ for some $k$ where 
$a_j, b_j \in S$ and $u_j \in {\rm Word}(S \cup S^{-1})$ for each $j=1, \dots, k$ 
such that $w$ and $W$ represent the same element in the free group. 
   Prepare a noose system consisting of $k$ nooses 
$N_1^\ast, \dots, N_k^\ast$ in the head $B$ of $N$ whose base point is the neck of $N$. 
Let $A^\ast$ be a regular neighborhood in $B$ of the union of the nooses $N_1^\ast, \dots, N_k^\ast$. 
    Apply the same argument with the proof of Lemma~\ref{thm:DiagramSpecial}  
to the small nooses $N_1^\ast, \dots, N_k^\ast$ and the word 
$W= (a_1, b_1)^{u_1} \cdots (a_k, b_k)^{u_k}$, 
we define a colored  doodle diagram $D$ over $A^\ast$, denoted by $D \cap A^\ast$, such that the intersection word along $\partial A^\ast$ is the word $W$.  Note that $N_1^\ast, \dots, N_k^\ast$ are proper nooses for this $D \cap A^\ast$. 
   Since $W$ and $w$ represent the same element in the free group, using Lemma~\ref{thm:BasicObservation}, 
   we can extend the diagram $D \cap A^\ast$ to a diagram over $B$, denoted by $D \cap B$, by adding some simple arcs with labels in $S$ and with orientations such that the intersection word along $\partial B$, the loop of $N$, is the word $w$.  
For every $i=1, \dots, m$, apply this argument and we construct a doodle diagram over the heads of $N_1, \dots,N_m$.  
   Let $A$ be the union of the heads and $m$ bands along the ropes connecting the heads to a neighborhood of the base point.  
   Now we have a colored doodle diagram over $A$, denoted by $D \cap A$,  such that the intersection word along $\partial A$ is $w_1 \dots w_m$, and the intersection word along $N_i$ is $w_i$ for each $i=1, \dots, m$. 
Since the word $w_1 \dots w_m$ represents the identity element of the free group, 
by Lemma~\ref{thm:BasicObservation}, 
we can extend the diagram $D \cap A$ to a colored doodle diagram $D$ in $S^2$ by adding some simple arcs in the closure of $S^2 \setminus A$.  Then the diagram $D$ and the noose system $N_1, \dots, N_m$ satisfy that $I(D, {\cal N})$ is 
$w_1 \dots w_m \equiv 1$.  
\end{proof}

The method introduced above Lemma~\ref{thm:DiagramSpecial} is sometimes not so effective to obtain a simple doodle diagram.  For example, we can consider immersed arcs with labels and orientations for the heads as in the left side of Figure~\ref{dfigBDab} instead on those in Figure~\ref{dfigBCab}.  The intersection word along the boundary of $A$ is still 
$((a,b),c^a) \, ((c,a),b^c) \, ((b,c),a^b)$, and we can extend it to a colored doodle diagram 
 as in the right hand side of the figure.  
Then the doodle diagram and the noose system yield the same commutator identity.

    \begin{figure}[h]
    \centerline{\epsfig{file=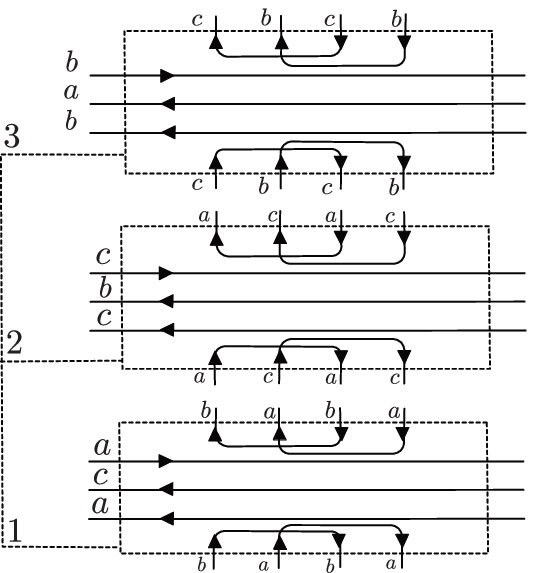, height=5.5cm} 
    \qquad \qquad 
    \epsfig{file=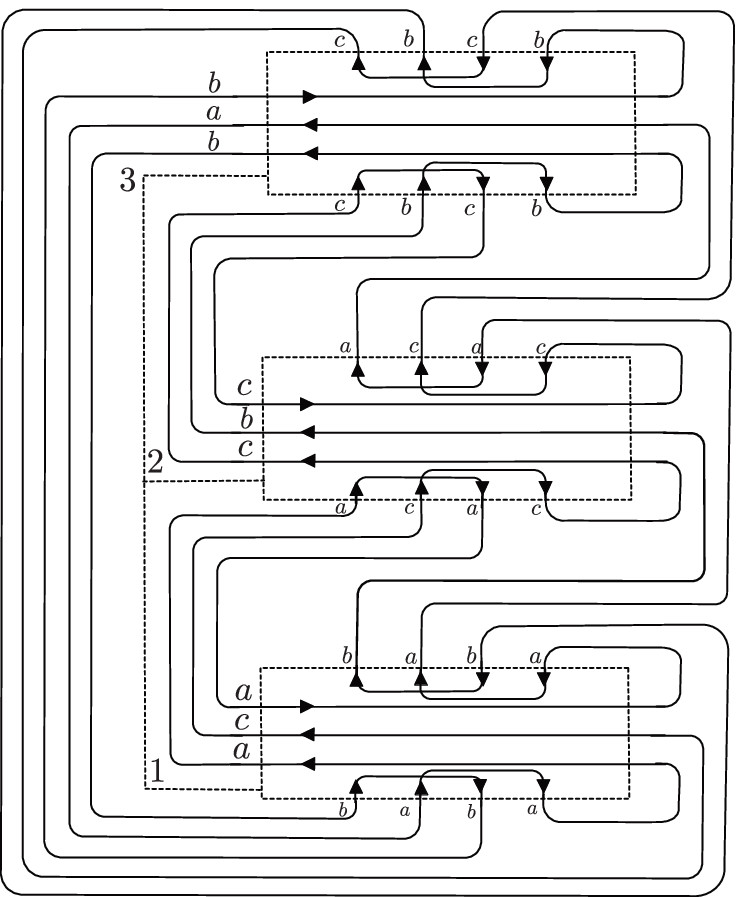, height=8cm} }
    \caption{Construction of a diagram}\label{dfigBDab}
    \end{figure}

Note that the number of crossings of the diagram in the former (Figure~\ref{dfigBCab}) is $36$ and that of the latter (Figure~\ref{dfigBDab}) is $6$.  In this sense the latter one is simpler than the former.  Moreover, the latter diagram is a minimal diagram, i.e., there are no monogons and no bigons.  It is known that a minimal diagram with $6$ crossings is a Borromean doodle diagram \cite{BFKKdoodles}.  By an ambient isotopy, we can deform the diagram with the noose system  in Figure~\ref{dfigBDab} 
into the standard diagram of the Borromean doodle with a noose system as in Figure~\ref{dfigBE}.   
Thus we see that the Hall-Witt identity can be obtained from the Borromean doodle. 
(The labels $b$ and $c$ in Figure~\ref{dfigBE} differ from those in Figure~\ref{dfigBorromean1dfigBorromean2}. When we switch $b$ and $c$ in Figure~\ref{dfigBE}, we have the Hall-Witt identity with letters $b$ and $c$ switched.)

    \begin{figure}[h]
    \centerline{\epsfig{file=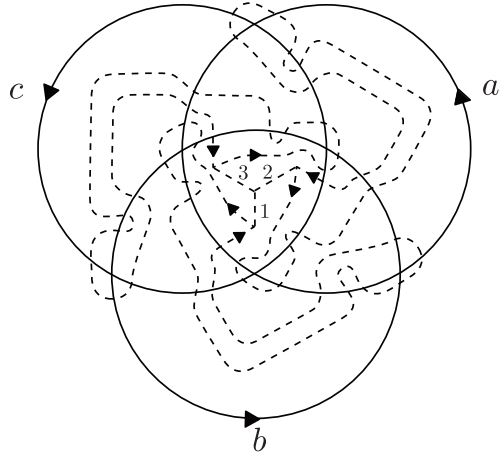, height=6cm} }
    \caption{Construction of a diagram}\label{dfigBE}
    \end{figure}

\section{Equivalence relations on elementary commutator identities}
\label{sect:EquivalnceIdentities}

In this section, we introduce fundamental transformations, including the action of the braid group, and define 
three kinds of equivalence relations on elementary commutator identities: 
strict equivalence $\cong$, equivalence $\simeq$ and weak equivalence $\sim$.  
We show that a colored doodle diagram induces a unique elementary commutator identity up to equivalence $\simeq$ (Theorem~\ref{thm:41})  and that a colored  doodle  induces a unique elementary commutator identity up to equivalence $\sim$ (Theorem~\ref{thm:42}).

We call the transformations (I)--(VII)  listed below {\em fundamental transformations}. 
Let $w_1 \cdots w_m \equiv 1$ be an elementary commutator identity on $S$.

\begin{itemize}

\item[(I)] (cyclic permutation) 
$$ \rho: w_1 w_2 \cdots w_m \equiv 1 \mapsto w_2 \cdots w_m w_1 \equiv 1. $$

\item[(II)] (braid action or Hurwizt  action)  
Let $j \in \{1, \dots, m-1\}$. 
$$ \sigma_j :  w_1 \cdots w_j w_{j+1} \cdots w_m \equiv 1 
\mapsto 
w_1 \cdots w_{j+1}  w_j^{w_{j+1}} \cdots w_m \equiv 1.  $$
The inverse of $\sigma_j$ is given by 
$$ \sigma_j^{-1} :  w_1 \cdots w_j w_{j+1} \cdots w_m \equiv 1 
\mapsto 
w_1 \cdots w_{j+1}^{w_j^{-1}}  w_j \cdots w_m \equiv 1.  $$

\item[(III)] (simultaneous conjugation or global conjugation) 
$$ {\rm conj}(u) :  w_1 w_2 \cdots  w_m \equiv 1 
\mapsto 
w_1^u w_2^u \cdots  w_m^u \equiv 1,  $$
where $u \in {\rm Word}(S \cup S^{-1})$.  

\item[(IV)] (changing local conjugation I) 
Let $i \in \{1, \dots, m\}$. 
$$ w_1 \cdots w_i \cdots w_m \equiv 1 
\mapsto 
w_1  \cdots w_i' \cdots w_m \equiv 1,  $$
where $w_i = (a_i, b_i)^{u_i}$ 
and $w_i' = (a_i, b_i)^{u_i'}$ 
such that $u_i$ and $u_i'$ represent the same element in the free group on $S$.   

\item[(V)] (changing local conjugation II) 
Let $i \in \{1, \dots, m\}$. 
$$ w_1 \cdots w_i \cdots w_m \equiv 1 
\mapsto 
w_1  \cdots w_i' \cdots w_m \equiv 1,  $$
where $w_i = (a_i, b_i)^{u_i}$ 
and $w_i' = (a_i, b_i)^{u_i'}$ 
such that $u_i' =  (a_i, b_i) u_i$ or $u_i' =  (a_i, b_i)^{-1} u_i$.   

\item[(VI)] (insertion/deletion of a trivial commutator) 
Insert or delete $(a,a)^u$ and change the length $m$ by one, 
where $a \in S$ and  $u \in {\rm Word}(S \cup S^{-1})$.  

\item[(VII)] (insertion/deletion of a cancelling pair) 
Insert or delete $(a,b)^u (b,a)^u$ and change the length $m$ by two, 
where $a, b \in S$ and $u \in {\rm Word}(S \cup S^{-1})$.  

\end{itemize}
  
Two elementary commutator identities on $S$ are {\em strictly equivalent}, {\em equivalent}, or {\em weakly equivalent} if they are related by a finite sequence of transformations (IV)--(V), 
(I)--(V), or (I)--(VII) respectively, and  
we denote the equivalence relation by $\cong$, $\simeq$ or $\sim$ respectively.

\begin{theorem}\label{thm:41}
A colored doodle diagram induces a unique elementary commutator identity up to equivalence $\simeq$.  Namely, 
let $D$ be a colored doodle diagram 
and let ${\cal N}$ and ${\cal N}'$  be proper noose systems for $D$.  Then $I(D, {\cal N})  \simeq  I(D, {\cal N}')$.  
\end{theorem}

\begin{theorem}\label{thm:42}
A colored doodle  induces a unique elementary commutator identity up to equivalence $\sim$. Namely, 
let  $D$ and $D'$ be colored doodle diagrams representing the same colored doodle, and   
let ${\cal N}$ and ${\cal N}'$ be proper noose systems for them.   
Then $I(D, {\cal N})  \sim  I(D', {\cal N}')$.  
\end{theorem}

We devote this section to proving these theorems.  

Let $O$ be the origin of $\R^2 = S^2 \setminus \{\infty\}$. For a while, we consider a case that a colored doodle diagram $D$ avoids $O$ and $\infty$, and the base points of noose systems are $O$.  

Let $\Sigma$ be a set of $m$ points of $S^2$ avoiding $O$ and $\infty$.  
A {\em proper noose for $\Sigma$} means a noose $N$ such that 
$N \cap \Sigma$  is an interior point of the head.  
A {\em proper noose system for $\Sigma$} is a noose system 
consisting of  proper nooses for $\Sigma$ such that 
every point of $\Sigma$ is contained in a head.  
We say that  two proper nooses $N$ and $N'$ for $\Sigma$ with root $O$ 
 are {\it homotopic} or {\em homotopic in $S^2$ with respect to $\Sigma$} if there is a homotopy of embeddings  $f_s : X =  D^2 \cup I' \to S^2 $ ($s \in [0,1]$)  satisfying the following. 
\begin{itemize}
\item $N= f_0(X)$ and $N'= f_1(X)$.  
\item  For each $s$, $f_s((2,0)) =O$.  
\item For each $s$, $f_s(X) \cap \Sigma$ is an interior point of the head $f_s(D^2)$. 
\end{itemize}
Here $f_s(X)$ may intersect with $\infty$, although we are assuming that $f_0(X)$ and $f_1(X)$ are away from $\{\infty\}$. 

Two proper noose systems ${\cal N}= (N_1, \dots, N_m)$ and ${\cal N}'= (N_1', \dots, N_m')$ for $\Sigma$ with base point $O$ 
 are {\it homotopic} or {\em homotopic in $S^2$ with respect to $\Sigma$} if each $N_i$ is homotopic to $N_i'$.

Let $D$ be a doodle diagram avoiding $O$ and $\infty$, and let $\Sigma(D)$ be the set of crossings.  
Note that a proper noose system for $D$ is a proper noose system for $\Sigma(D)$.  
The converse is not true in general. However, when a proper noose system for $\Sigma(D)$ with base point $O$ is given, moving the nooses by a homotopy in $S^2$ with respect to $\Sigma(D)$ we obtain a (non-unique) proper noose system for $D$.

Let $D$ be a colored doodle diagram avoiding $O$ and $\infty$, and let  
 ${\cal N}= (N_1, \dots, N_m)$ be a proper noose system for $\Sigma(D)$ with base point $O$. 
A {\em commutator identity obtained from $D$ by using ${\cal N}$} means a commutator identity  
$I(D, {\cal N}')$ obtained from $D$ by using a proper noose system ${\cal N}'$ for $D$ which is homotopic to ${\cal N}$.  
It is denoted by 
$$I(D, {\cal N}).$$

\begin{lemma}\label{lem:homotopychange}
In the situation above, $I(D, {\cal N})$ is well-defied up to strict equivalence $\cong$. Precisely speaking, 
let ${\cal N'}= (N_1', \dots, N_m')$ and ${\cal N}''= (N_1'', \dots, N_m'')$ be proper noose systems for $D$ with base point $O$  such that they are 
homotopic  to  ${\cal N}= (N_1, \dots, N_m)$ in $S^2$ with respect to $\Sigma(D)$.  
Then the commutator identities $I(D, {\cal N}')$  and $I(D, {\cal N}'')$ are strictly equivalent.  
\end{lemma}

\begin{proof}
Note that ${\cal N'}$ is homotopic to ${\cal N}''$ in $S^2$ with respect to $\Sigma(D)$.  Fix $i \in \{1, \dots, m\}$ and consider a homotopy $f_s: X \to S^2$ ($s \in [0,1]$)  moving $N_i'$ to $N_i''$.  
Taking such a homotopy such that the heads $f_s(D^2)$ are small, we may assume that for each $s \in [0,1]$, the intersection $f_s(D^2) \cap D$ is the union of two embedded arcs in the head $f_s(D^2)$ intersecting each other on a single crossing of $D$.  Note that  
the intersection word along the loop of $N_i'$ is the same with that of $N_i''$.  
If the neck of $N_i'$ is disjoint from the diagram $D$ through the 
homotopy moving $N_i'$ to $N_i''$, then  the intersection word along the  rope of $N_i'$ and that of $N_i''$ represent the same element in the free group on $S$, since the homotopy avoids $\Sigma(D)$.  Thus, the commutator identities are related by transformations (IV).   
If the homotopy moves the neck of $N_i'$ around the crossing of $D$ 
several times and keeps the remaining part of the noose fixed, then the commutator  identities are related by transformations  (V).  
In general case, a homotopy moving $N_i'$ to $N_i''$ is combined with these homotopies, and hence the commutator  identities are related by transformations  (IV) and (V).  
\end{proof}

Let $S^1$ be the unit circle of $\R^2$ and let $Q = \{ q_1, \dots, q_m\}$ be a fixed $m$ points 
on $S^1$ evenly arranged counterclockwise  in this order.  

Let $D_-$ be a small round $2$-disk in $S^2 = \R^2 \cup \{ \infty\}$ with center $O$, and let $D_+$ be the complementary $2$-disk in $S^2$ which is the closure of $S^2 \setminus D_-$.  

Consider the mapping class group ${\rm Mod}(D_+, Q)$, that is the group of isotopy classes of self-homeomorphisms of $D_+$ sending $Q$ to itself whose restriction to the boundary of $D_+$ is the identity. 
(We often use the same symbol for an element of ${\rm Mod}(D_+, Q)$ and its representative. 
For two elements $g$ and $h$, the product $gh$ is defined by the composition $g \circ h$.) 
 It is well known that the mapping class group is identified with the $m$-braid group on $Q$ in $D_+$, and it is generated by $\tau_1, \dots, \tau_{m-1}$ where $\tau_i$ is a \lq\lq disk twist\rq\rq  which rotates the arc on $S^1$ between $q_i$ and $q_{i+1}$ counterclockwise in its regular neighborhood (cf. \cite{Birman}). 

Let  $g$ be an element of ${\rm Mod}(D_+, Q)$.  
For a proper noose system ${\cal N} =(N_1, \dots, N_m)$ for $Q$ with base point $O$, we denote by $g({\cal N})$ a proper noose system $(g(N_1), \dots, g(N_m))$ for $Q$ with base point $O$, which is well-defined up to  isotopy of $S^2$ keeping $D_-$ and $\Sigma$ fixed pointwise. In particular, it is well-defined, as a proper noose system,  up to homotopy in $S^2$ with respect to $\Sigma$.  
For a doodle diagram $D$ avoiding $O$ and $\infty$, we denote by $g(D)$ a doodle diagram obtained from $D$ by $g$, which is well-defined up to isotopy of $S^2$ keeping $D_-$ and $\Sigma$ fixed pointwise.  

A {\em standard noose system} for $Q$ is a proper noose system ${\cal N}^0 = 
(N_1^0, \dots, N_m^0)$ for $Q$ with base point $O$ such that for each $i$, the head of $N_i^0$ is a small  round $2$-disk with center $q_i$ and the rope is a radius connecting $O$ to $N_i^0$. See Figure~\ref{dfigStandard} where $m=6$. 

    \begin{figure}[h]
    \centerline{\epsfig{file=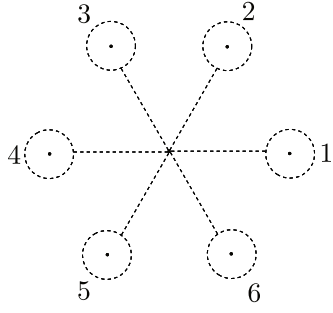, height=4cm} }
    \caption{A standard noose system for $Q$}\label{dfigStandard}
    \end{figure}

\begin{lemma}\label{lem:identityA} 
Let $D$ be a colored doodle diagram avoiding $O$ and $\infty$  with $\Sigma(D) =Q$, and   
let ${\cal N}$ be a proper noose system for $Q$ with base point $O$.  For any $g \in {\rm Mod}(D_+, Q)$, we have  
$$ I(g(D), g({\cal N}))  \cong  I(D, {\cal N})   \quad \mbox{and} \quad 
I(D, g({\cal N})) \cong I(g^{-1}(D), {\cal N}).$$ 
\end{lemma}

\begin{proof}
The first relation is trivial by definition and Lemma~\ref{lem:homotopychange}.  
The second one is obtained from the first by replacing $D$ with $g^{-1}(D)$.   
\end{proof}

\begin{lemma}\label{lem:standardA} 
Let $D$ be a colored doodle diagram avoiding $O$ and $\infty$ with $\Sigma(D) =Q$.  
Let ${\cal N}^0 = (N_1^0, \dots, N_m^0)$ be the standard noose system for $Q$.  
\begin{itemize} 
\item[(1)] 
$$ I(D, \tau_j^{\epsilon}({\cal N}^0)) \cong \sigma_j^{-\epsilon} I(D, {\cal N}^0),$$ 

\item[(2)] 
$$ I(D,  \tau_{j_n}^{\epsilon_n} \dots \tau_{j_1}^{\epsilon_1}({\cal N}^0))  
\cong \sigma_{j_1}^{-\epsilon_1} \cdots  \sigma_{j_n}^{-\epsilon_n} I(D, {\cal N}^0).$$ 

\end{itemize}
\end{lemma}

\begin{proof}
(1) 
Suppose that $\epsilon=1$.  
The  $j$th and the $(j+1)$st nooses of ${\cal N}_0$ are as in the left of Figure~\ref{dfigCab} and those 
of $\tau_j({\cal N}^0)$ are as in the right  of the figure, up to homotopy with respect to $Q$.  
If the intersection words of the $j$th and the $(j+1)$st nooses of ${\cal N}_0$ 
 are $w_j$ and $w_{j+1}$ then those of $g({\cal N}^0)$ are $w_{j+1}^{w_j^{-1}}$ and $w_j$.  
 Similarly, when $\epsilon =-1$, they are $w_j$ and $w_{j+1}^{w_j}$.  
 Thus, $I(D, \tau_j^{\epsilon}({\cal N}^0)) \cong 
 \sigma_j^{-\epsilon} I(D, {\cal N}^0)$. 
This completes the proof of (1).   

    \begin{figure}[h]
    \centerline{\epsfig{file=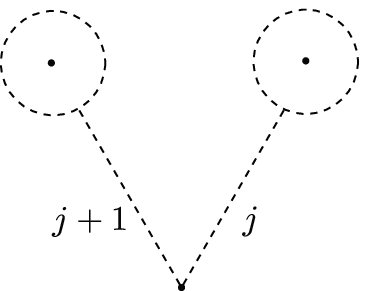, height=3.0cm} 
    \qquad \qquad 
    \epsfig{file=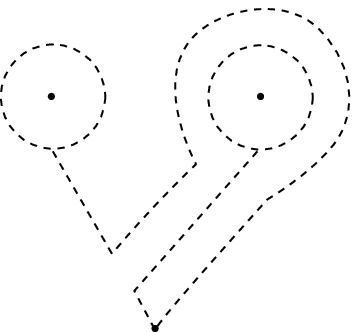, height=3.5cm} }
    \vspace*{8pt}
    \caption{The $j$th and $(j+1)$st   nooses of  ${\cal N}_0$, and those of $\tau_j({\cal N}_0)$}\label{dfigCab}
    \end{figure}

(2) 
By Lemma~\ref{lem:identityA} we have $  
I(D, \tau_j^{\epsilon}({\cal N}^0)) \cong I( \tau_j^{-\epsilon}(D), {\cal N}^0)$.  
Combining this and the equivalence  in (1), we obtain   
$$
 I( \tau_j^{-\epsilon}(D), {\cal N}^0) \cong 
\sigma_j^{-\epsilon} I(D, {\cal N}^0)
$$
which holds for any colored doodle diagram $D$ avoiding $O$ and $\infty$ with $\Sigma(D)=Q$.  
Applying this inductively, we have 
\begin{eqnarray*}
I(\tau_{j_1}^{-\epsilon_1} \tau_{j_2}^{-\epsilon_2} \dots \tau_{j_n}^{-\epsilon_n}(D), {\cal N}^0)   
& \cong &
 \sigma_{j_1}^{-\epsilon_1} I(\tau_{j_2}^{-\epsilon_2} \dots \tau_{j_n}^{-\epsilon_n}(D), {\cal N}^0)   \\ 
& \cong &
 \sigma_{j_1}^{-\epsilon_1} \sigma_{j_2}^{-\epsilon_2} 
 I(\tau_{j_3}^{-\epsilon_3} \dots \tau_{j_n}^{-\epsilon_n}(D), {\cal N}^0)   \\ 
  & \cong & \cdots \\ 
 & \cong &
 \sigma_{j_1}^{-\epsilon_1} \sigma_{j_2}^{-\epsilon_2} \dots \sigma_{j_n}^{-\epsilon_n}
 I(D, {\cal N}^0).
\end{eqnarray*}
By Lemma~\ref{lem:identityA} again, we have 
$$ I(D,  \tau_{j_n}^{\epsilon_n} \dots \tau_{j_1}^{\epsilon_1}({\cal N}^0))  
\cong \sigma_{j_1}^{-\epsilon_1}  \cdots  \sigma_{j_n}^{-\epsilon_n} I(D, {\cal N}^0).$$ 
\end{proof}

\begin{lemma}\label{lem:standardB} 
Let $D$ be a colored doodle diagram avoiding $O$ and $\infty$  with $\Sigma(D) =Q$.  
Let ${\cal N}$ and ${\cal N}'$ be proper noose systems for $Q$ with base point $O$.  Then 
$I(D, {\cal N}) \simeq I(D, {\cal N}')$.  
\end{lemma} 

\begin{proof}
There is an element  $ g \in {\rm Mod}(D_+,Q)$ 
such that  ${\cal N} = g ({\cal N}^0)$  up to homotopy in $S^2$ with respect to $Q$.  
By Lemma~\ref{lem:standardA}, $I(D, {\cal N}) \cong I(D, g({\cal N}^0)) \simeq I(D, {\cal N}^0)$. 
Similarly, $I(D, {\cal N}') \simeq I(D, {\cal N}^0)$. Thus, $I(D, {\cal N}) \simeq I(D, {\cal N}')$. 
\end{proof}

\noindent 
{\bf Proof of Theorem~\ref{thm:41}.}
(1) 
First we consider a case that ${\cal N}$ and ${\cal N}'$ have the same base point, say $\ast$. 
Take a homeomorphism $f : S^2 \to S^2$ such that $f(\ast) = O$, $f(\Sigma(D))= Q$ and $f(D)$ is away from $\infty$.  
Put $f(D) = \tilde D$, $f({\cal N})= \tilde{\cal N}$ and $f({\cal N}') = \tilde{\cal N}'$.  Then 
$I(D, {\cal N}) = I( \tilde{D}, \tilde{\cal N})$ and 
$I(D, {\cal N}') = I( \tilde{D}, \tilde{\cal N}')$.  
Applying  Lemma~\ref{lem:standardB} to $\tilde{D}$, $\tilde{\cal N}$ and $\tilde{\cal N}'$, we have $I( \tilde{D}, \tilde{\cal N}) \simeq I( \tilde{D}, \tilde{\cal N}')$.  Thus,  
$I(D, {\cal N}) \simeq I(D, {\cal N}')$. 

(2) Consider a case that the base point $\ast$ of ${\cal N}$  and the base point $\ast'$ of ${\cal N}'$ are contained in the same region, say $R$,  of $S^2 \setminus D$.  Take an open set $U$ in $R$ homeomorphic to the open unit $2$-disk with $\{ \ast, \ast'\} \subset U$ such that the closure $\overline{U}$ is in $R$.   
There is a homeomorphism $f: S^2 \to S^2$ such that $f(\ast) = \ast'$ and the support of $f$ is in $U$. 
Then $I(D, {\cal N}) = I(f(D), f({\cal N}))= I(D, f({\cal N}))$.  Since $f({\cal N})$ has the same base point with ${\cal N}'$, by (1) we have  $I(D, f({\cal N})) \simeq I(D, {\cal N}')$.  Thus, 
$I(D, {\cal N}) \simeq I(D, {\cal N}')$. 

(3) Consider a case that the base point $\ast$ of ${\cal N}$  and the base point $\ast'$ of ${\cal N}'$ are contained in different regions of $S^2 \setminus D$.  It is sufficient to consider a case that the regions containing $\ast$ and $\ast'$ are adjacent and there is a simple arc $\gamma$ in $S^2$ connecting $\ast$ and $\ast'$ such that $\gamma$ intersects with $D$ transversely on a single point.  Let $B$ be a regular neighborhood of $\gamma$ in $S^2$, which is a $2$-disk 
containing $\ast$ and $\ast'$ such that $B \cap D$ is a simple proper arc in $B$ separating $\ast$ and $\ast'$. 
By (1), without loss of generality, we may assume that ${\cal N}$ has base point $\ast$ and $N_i \cap \gamma=\{\ast\}$ for every noose $N_i$ of ${\cal N}$.  
     Let ${\cal N}''$ be a proper noose system for $D$ obtained from ${\cal N}$ by a homotopy in $S^2$ moving the base point $\ast$ to $\ast'$ in $B$ along $\gamma$.  Then $I(D, {\cal N}'') = {\rm conj}(a^{\epsilon}) I(D, {\cal N})$ where $a$ is the label of the component of $D$ whose restriction to $D$ is the proper arc separating $\ast$ and $\ast'$ and the sign $\epsilon \in \{\pm 1\}$ is determined from its  orientation.  Thus, 
$I(D, {\cal N}) \simeq I(D, {\cal N}'')$.  By (2), 
$I(D, {\cal N}'') \simeq I(D, {\cal N}')$.  Thus, 
$I(D, {\cal N}) \simeq I(D, {\cal N}')$. 
     This completes the proof of Theorem~\ref{thm:41}. 
    \vskip-\baselineskip\prbox\par
	\addvspace{12pt plus3pt minus3pt}

\vspace{0.3cm}
\noindent 
{\bf Proof of Theorem~\ref{thm:42}.}
By Theorem~\ref{thm:41}, we see that the equivalence class of $I(D, {\cal N})$ does not depend on a choice of ${\cal N}$ and the isotopy class of $D$ in $S^2$.  Thus, it is sufficient to prove that if $D'$ is obtained from $D$ by a move $H_1^{-1}$ or $H_2^{-1}$ then 
$I(D, {\cal N}) \sim I(D, {\cal N}')$ for some noose systems ${\cal N}$ and ${\cal N}'$ for them. 

Suppose that $D'$ is obtained from $D$ by $H_1^{-1}$.  Take a base point $\ast$ avoiding $D$ and $D'$ and the area where the $H_1^{-1}$ move is applied.  Let $N_1$ be a proper noose for $D$ whose head containing the crossing of $D$ which is removed by the $H_1^{-1}$ move such that the intersection word of $N_1$ with $D$ is $(a,a)^u$ for some $a \in S$ and $u \in {\rm Word}(D)$. We can take a proper noose system ${\cal N} =(N_1, N_2, \dots, N_m)$ for $D$ including $N_1$.   
Let ${\cal N}'= (N_2, \dots, N_m)$, which is a proper noose system for $D'$.  
Then for each $i=2, \dots, m$ the intersection word of $N_i$ with $D$ is the same with that with $D'$, hence 
$I(D', {\cal N}')$ is obtained from $I(D, {\cal N})$ by a transformation (VI).  

Suppose that $D'$ is obtained from $D$ by $H_2^{-1}$.  Take a base point $\ast$ avoiding $D$ and $D'$ and the area where the $H_2^{-1}$ move is applied.  Let $N_1$ and $N_2$ be proper nooses for $D$ whose heads containing the crossings of $D$ which are removed by the $H_2^{-1}$ move such that the intersection words of them with $D$ are $(a,b)^u$ and $(b,a)^u$ for some $a, b \in S$ and $u \in {\rm Word}(D)$. We can take a proper noose system ${\cal N} =(N_1, N_2, N_3, \dots, N_m)$ for $D$ including $N_1$ and $N_2$.  
Let ${\cal N}'= (N_3, \dots, N_m)$, which is a proper noose system for $D'$.  
Then for each $i=3, \dots, m$ the intersection word of $N_i$ with $D$ is the same with that with $D'$, and hence 
$I(D', {\cal N}')$ is obtained from $I(D, {\cal N})$ by a transformation (VII).  
    This completes the proof of Theorem~\ref{thm:42}. 
    \vskip-\baselineskip\prbox\par
	\addvspace{12pt plus3pt minus3pt}

\section{Cobordisms of colored oriented doodles}
\label{sect:Cobordism}

We discuss cobordisms of colored oriented doodles.  

Two colored doodle diagrams $D$ and $D'$ are {\it cobordant} if they are related by a finite sequence of ambient isotopies in $S^2$, moves $H_1^{\pm 1}$ and $H_2^{\pm 1}$, 
and the following local moves:   
\begin{itemize}
\item[(1)] Insertion or deletion of a trivial component with a color in the coloring set $S$. 
Here a {\em trivial component} is a simple loop which is disjoint from the other components of the doodle diagram. See the left hand side of Figure~\ref{dfigCab12} (It is also called a floating component.)  
\item[(2)] (A bridge move) Replacement as in the right hand side of Figure~\ref{dfigCab12} between arcs of the same color. 
\end{itemize} 

Two colored oriented doodles are {\em cobordant} if their representatives are cobordant.

    \begin{figure}[h]
    \centerline{\epsfig{file=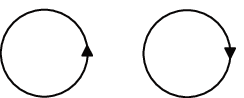, height=1.5cm} 
    \qquad \qquad \qquad 
    \epsfig{file=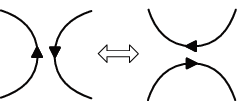, height=1.5cm} }
    \vspace*{8pt}
    \caption{Trivial components and a bridge move}\label{dfigCab12}
    \end{figure}

The following is a key lemma of this section. 

\begin{lemma}\label{thm:CobA}
Let $D$ and $D'$ be colored doodles diagrams such that 
$I(D, {\cal N}) \sim I(D', {\cal N}')$ for some proper noose systems 
${\cal N}$ and ${\cal N}'$.  Then $D$ and $D'$ are cobordant.  
\end{lemma}

\begin{proof} 
(1) We first prove that if $I(D, {\cal N}) = I(D', {\cal N}')$ for some proper noose systems 
${\cal N}$ and ${\cal N}'$ then $D$ and $D'$ are cobordant. 

Moving $D$ and ${\cal N}$ together by an isotopy of $S^2$, we may assume that ${\cal N}= {\cal N}'$.  
Let $A$ be a regular neighborhood of the union of nooses of ${\cal N}$ and let $E$ be the closure of $S^2 \setminus A$.  Both $A$ and $E$ are homeomorphic to a $2$-disk.
Since $I(D, {\cal N}) = I(D', {\cal N}')$, moving $D$ in a regular neigbourhood of $A$, we may assume that $D\cap A= D'\cap A$. The intersection $D \cap E$ consists of some properly embedded arcs and some (or no) embedded loops. So does $D' \cap E$. Applying a finite sequence of local moves of insertion/deletion of a trivial component and bridge moves in $E$, we can transform $D \cap E$ to $D' \cap E$. Thus $D$ is cobordant to $D'$.  

(2) We prove that if $I(D, {\cal N}) \cong I(D', {\cal N}')$ for some proper noose systems 
${\cal N}$ and ${\cal N}'$ then $D$ and $D'$ are cobordant. 

(changing local conjugation I)  
Suppose that, for some $i \in \{1, \dots, m\}$, $I(D, {\cal N})$ is $w_1 \cdots w_i \cdots w_m \equiv 1$ and 
$I(D', {\cal N}')$ is $w_1 \cdots w_i' \cdots w_m \equiv 1$ where 
$w_i = (a_i, b_i)^{u_i}$ 
and $w_i' = (a_i, b_i)^{u_i'}$ 
such that $u_i$ and $u_i'$ represent the same element in the free group on $S$.  
Applying a finite sequence of insertion/deletion of a trivial component and bridge moves 
in a regular neighborhood of the rope of the $i$th noose, we can change $D$ to $D''$ such that 
$D$ and $D''$ are cobordant and $I(D'', {\cal N})= I(D', {\cal N}')$.  
By (1), $D''$ is cobordant to $D'$.   Thus $D$ is cobordant to $D'$.  

(changing local conjugation II) 
Suppose that, for some $i \in \{1, \dots, m\}$, 
$I(D, {\cal N})$ is $w_1 \cdots w_i \cdots w_m \equiv 1$ and 
$I(D', {\cal N}')$ is $w_1 \cdots w_i' \cdots w_m \equiv 1$ 
where $w_i = (a_i, b_i)^{u_i}$ 
and $w_i' = (a_i, b_i)^{u_i'}$ 
such that $u_i' =  (a_i, b_i) u_i$ or $u_i' =  (a_i, b_i)^{-1} u_i$. 
Let ${\cal N}''$ be a proper noose system for $D$ obtained from   ${\cal N}$ by rotating the neck of the $i$th head along its loop counterclockwise or clockwise so that 
$I(D, {\cal N}'') = I(D', {\cal N}')$.  By (1), $D$ is cobordant to $D'$. 

Therefore, we see that if $I(D, {\cal N}) \cong I(D', {\cal N}')$ for some proper noose systems 
${\cal N}$ and ${\cal N}'$ then $D$ and $D'$ are cobordant. 

(3) We prove that if $I(D, {\cal N}) \simeq I(D', {\cal N}')$ for some proper noose systems 
${\cal N}$ and ${\cal N}'$ then $D$ and $D'$ are cobordant.

(cyclic permutation) 
Suppose that $I(D', {\cal N}') = \rho I(D, {\cal N})$.  Let ${\cal N} = (N_1, \dots, N_m)$ and  
put ${\cal N}'' =(N_2, \dots, N_m, N_1)$.  
Then $I(D, {\cal N}'')= \rho I(D, {\cal N}) = I(D', {\cal N}')$.  By (1), $D$ and $D'$ are cobordant.  

(braid action) 
Suppose that $I(D', {\cal N}') = \sigma_j I(D, {\cal N})$. 
Moving $D$ and ${\cal N}$ by an isotopy of $S^2$ and moving $D'$ and ${\cal N}'$ be an isotopy of $S^2$ respectively, we may assume that $\Sigma(D) = \Sigma(D') =Q$ and ${\cal N} = {\cal N}^0$, the standard noose system. 
Let ${\cal N}'' = \tau_j^{-1} ({\cal N})$. Then $I(D, {\cal N}'') \cong \sigma_j I(D, {\cal N})$ by Lemma~\ref{lem:standardA}~(1). Thus, $I(D, {\cal N}'') \cong I(D',  {\cal N}')$.  By (2),  $D$ and $D'$ are cobordant.  

(simultaneous conjugation) 
Suppose that $I(D', {\cal N}') =  {\rm conj}(u) I(D, {\cal N})$ 
where $u \in {\rm Word}(S \cup S^{-1})$.  
Let $D''$ be a colored doodle diagram obtained from $D$ by a finite sequence of insertion of a trivial component whose center is the base point of ${\cal N}$ such that $D$ and $D''$ are cobordant and 
$I(D'', {\cal N}) = {\rm conj}(u) I(D, {\cal N})$.  Thus, $I(D', {\cal N}')= I(D'', {\cal N})$. 
By (1), $D'$ and  $D''$ are cobordant, and hence $D$ and $D'$ are cobordant. 

Now we see that if $I(D, {\cal N}) \simeq I(D', {\cal N}')$ for some proper noose systems 
${\cal N}$ and ${\cal N}'$ then $D$ and $D'$ are cobordant. 

(4) We prove that if $I(D, {\cal N}) \sim I(D', {\cal N}')$ for some proper noose systems 
${\cal N}$ and ${\cal N}'$ then $D$ and $D'$ are cobordant.

(insertion/deletion of a trivial commutator) 
Suppose that $I(D', {\cal N}')$ is obtained from $I(D, {\cal N})$ by inserting  
$(a,a)^u$ where $a \in S$ and  $u \in {\rm Word}(S \cup S^{-1})$.   
Let $D^\ast$ be a colored doodle diagram 
consisting of an immersed loop with a single crossing with label $a$ surrounded by some simple loops 
such that a proper noose for it, say $N$, has the intersection word $(a,a)^u$.  See the left hand side of Figure~\ref{dfigCab34}.  (Note that $D^\ast$ is cobordant to the empty diagram.)
 Let $D''$ be a colored doodle diagram and ${\cal N}''$ a proper noose system for $D''$ such that 
they are obtained from $D$ and ${\cal N}$ by inserting $D^\ast$ and $N$ near the base point of ${\cal N}$ 
such that $I(D'', {\cal N}'') = I(D', {\cal N}')$.  Note that $D$ is cobordant to $D''$. 
By (1), $D''$ is cobordant to $D'$.  Thus, $D$ and $D'$ are cobordant. 

(insertion/deletion of a cancelling pair) 
Suppose that $I(D', {\cal N}')$ is obtained from $I(D, {\cal N})$ by inserting  
 $(a,b)^u (b,a)^u$  
where $a, b \in S$ and $u \in {\rm Word}(S \cup S^{-1})$.  
Let $D^{\ast\ast}$ be a colored doodle diagram consisting of a pair of simple loops with labels $a$ and $b$ intersecting twice 
surrounded by some simple loops such that a pair of proper nooses for it, say $N_1^{\ast\ast}$ and $N_2^{\ast\ast}$, have the intersection words  $(a,b)^u$ and $(b,a)^u$.  See the right hand side of Figure~\ref{dfigCab34}. 
(Note that $D^{\ast\ast}$ is cobordant to the empty diagram.) 
 Let $D''$ be a colored doodle diagram and ${\cal N}''$ a proper noose system for $D''$ such that 
they are obtained from $D$ and ${\cal N}$ by inserting $D^{\ast\ast}$ and $N_1^{\ast\ast}$ and $N_2^{\ast\ast}$ near the base point of ${\cal N}$ such that $I(D'', {\cal N}'') = I(D', {\cal N}')$.  Note that $D$ is cobordant to $D''$. 
By (1), $D''$ is cobordant to $D'$.  Thus, $D$ and $D'$ are cobordant. 

Therefore we see that if $I(D, {\cal N}) \sim I(D', {\cal N}')$ for some proper noose systems 
${\cal N}$ and ${\cal N}'$ then $D$ and $D'$ are cobordant. 
\end{proof}

    \begin{figure}[h]
    \centerline{\epsfig{file=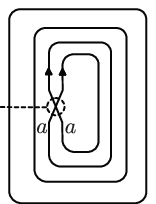, height=4.5cm} 
    \qquad \qquad 
    \epsfig{file=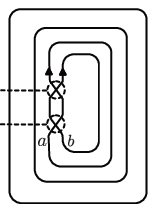, height=5.0cm} }
    \vspace*{8pt}
    \caption{$D^\ast$ and $D^{\ast\ast}$}\label{dfigCab34}
    \end{figure}

The following lemma is the converse of Lemma~\ref{thm:CobA}.  

\begin{lemma}\label{thm:CobB}
Let $D$ and $D'$ be colored oriented doodles diagrams, and let 
${\cal N}$ and ${\cal N}'$ be proper noose systems for them.  
If $D$ and $D'$ are cobordant then 
$I(D, {\cal N}) \sim I(D', {\cal N}')$.    
\end{lemma}

\begin{proof}
By Theorem~\ref{thm:42}, it is sufficient to consider a case that 
$D'$ is obtained from $D$ by a local move which is an insertion/deletion of a trivial component or a bridge move.  
Let ${\cal N}''$ be a proper noose system for $D$ such that every noose of it is away from the area where the local move takes place. Then ${\cal N}''$ is also a proper noose for $D'$ and 
$I(D, {\cal N}'') = I(D', {\cal N}'')$.  
Thus, by Theorem~\ref{thm:41}, we see that  $I(D, {\cal N}) \simeq I(D', {\cal N}')$.  
\end{proof}

By Lemmas~\ref{thm:CobA} and \ref{thm:CobB}, we have the following theorem.

\begin{theorem}\label{thm:CobC}
Let $D$ and $D'$ be colored doodles diagrams, and let 
${\cal N}$ and ${\cal N}'$ be proper noose systems for them.  
Then, $D$ and $D'$ are cobordant if and only if  
$I(D, {\cal N}) \sim I(D', {\cal N}')$.    
\end{theorem}

By Theorem~\ref{thm:42} we have a map from the set of colored doodles to the set of weak equivalence classes of elementary commutator identities.

\begin{theorem}\label{thm:CobD}
The map from the set of colored doodles to the set of weak equivalence classes of elementary commutator identities 
induces a bijection between 
cobordism classes of colored doodles and weak equivalence classes of elementary commutator identities.  
\end{theorem}

\begin{proof}
By Lemma~\ref{thm:CobB}, we see that 
the map from the set of colored doodles to the set of weak equivalence classes of elementary commutator identities factors through the set of cobordism classes of colored doodles.  
Lemma~\ref{thm:CobA} implies that this map is injective, and 
Theorem~\ref{thm:DiagramProper} implies that it is surjective. 
\end{proof}

In this paper we have discussed commutator identities related to doodles on the 2-sphere.  Doodles were generalized to surfaces with higher genus \cite{BFKKdoodles}.  Mark Culler \cite{Culler} studied commutator identities using surfaces.  We will discuss in a later paper commutator identities related to doodles on surfaces.

\end{document}